\documentclass[twoside]{ima}
\usepackage{amssymb,epsf,amsmath}

\usepackage{cite} 

\usepackage{graphicx,color}
\usepackage{booktabs}
\usepackage{enumerate}
\usepackage{subfigure}
\usepackage{ragged2e}
\usepackage{bm}

\newcommand{\N}{{\mathbb  N}}
\newcommand{\Z}{{\mathbb  Z}}
\newcommand{\Q}{{\mathbb  Q}}

\newcommand{\R}{{\mathbb  R}}
\let\ve=\mathbf
\newcommand\OS{Opt.}

\newtheorem{algorithm}{Algorithm}

\usepackage{ifpdf}

\newcommand\Order{\mathrm{O}}
\newcommand\st{\text{s.t.}}

\begin{document}
\title{ON THE COMPLEXITY OF NONLINEAR MIXED-INTEGER OPTIMIZATION}
\author{MATTHIAS K\"OPPE\thanks{University of California,
  Davis, Dept.~of
  Mathematics, One Shields Avenue, Davis, CA, 95616, USA,
  E-Mail: \texttt{mkoeppe@math.ucdavis.edu}}
}

\maketitle

\begin{abstract} 
  This is a survey on the computational complexity of nonlinear mixed-integer
  optimization.  It highlights a selection of important topics, ranging from
  incomputability results that arise from number theory and logic, to recently
  obtained fully polynomial time approximation schemes in fixed dimension, and to
  strongly polynomial-time algorithms for special cases.
\end{abstract}



\section{Introduction}

In this survey we study the computational complexity of nonlinear
mixed-integer optimization problems, i.e., models of the form
\begin{equation}
\begin{aligned}
  \hbox{max/min}\quad & f(x_1,\dots,x_n)\\
  \hbox{s.t.}\quad & g_1(x_1,\dots,x_n) \leq 0 \\
  & \quad\vdots \\
  & g_m(x_1,\dots,x_n) \leq 0 \\
  & \mathbf x\in\R^{n_1} \times \Z^{n_2},
\end{aligned} \label{eq:nonlinear-over-nonlinear}
\end{equation}
where $n_1+n_2=n$ and $f,g_1,\dots,g_m\colon \R^n\to\R$ are arbitrary nonlinear functions.

This is a very rich topic.  From the very beginning, questions such as how to
present the problem to an algorithm, and, in view of possible irrational
outcomes, what it actually means to solve the problem need to be answered.
Fundamental intractability results from number theory and logic on the one
hand and from continuous optimization on the other hand come into play.
The spectrum of theorems that we present ranges from incomputability
results, to hardness and inapproximability theorems, to  classes that
have efficient approximation schemes, or even polynomial-time or strongly
polynomial-time algorithms.\smallbreak

We restrict our attention to deterministic algorithms in the usual bit
complexity (Turing) model of computation.  Some of the material in the present
survey also appears in
\cite{hemmecke-koeppe-lee-weismantel:50-years-nonlinear-ip-chapter}.  For an
excellent recent survey focusing on other aspects of the complexity of
nonlinear optimization, including the performance of oracle-based models and
combinatorial settings such as nonlinear network flows, we
refer to Hochbaum \cite{hochbaum-2007:complexity-nonlinear}.  
We also do not cover the recent developments by Onn et
al. \cite{DeLoera+Hemmecke+Onn+Weismantel:08,Hemmecke+Onn+Weismantel:08,Lee+Onn+Weismantel:08,Lee+Onn+Weismantel:08b,Berstein+Lee+Maruri-Aguilar+Onn+Riccomagno+Weismantel+Wynn:08,Berstein+Onn:08,Berstein+Lee+Onn+Weismantel,DeLoera+Onn:2006a,DeLoera+Onn:2006b}
in the context of discrete convex
optimization, for which we refer to the
monograph~\cite{onn:nonlinear-discrete-monograph}.
Other excellent
sources are 
\cite{deKlerk2008} and 
\cite{pardalos:complexity-numerical-optimization}.

\clearpage
\section{Preliminaries}

\subsection{Presentation of the problem}

We restrict ourselves to a model where the problem is presented explicitly. 
In most of this survey, the functions $f$ and $g_i$ will be polynomial
functions presented in a sparse encoding, where all coefficients are rational
(or integer) and encoded in the binary scheme.  It is useful to assume that
the exponents of monomials are given in the unary encoding scheme; otherwise
already in very simple cases the results of function evaluations will have an
encoding length that is exponential in the input size.

In an alternative model, the functions are presented by oracles, such as
comparison oracles or evaluation oracles.  This model permits to handle more
general functions (not just polynomials), and on the other hand it is very
useful to obtain hardness results.

\subsection{Encoding issues for solutions}
\label{s:encoding}

When we want to study the computational complexity of these optimization
problems, we first need to discuss how to encode the input (the data of the
optimization problem) and the output (an optimal solution if it exists).  In
the context of \emph{linear} mixed-integer optimization, this is
straightforward: Seldom are we concerned with irrational objective functions
or constraints;  when we restrict the input to be rational as is usual, then also optimal
solutions will be rational.

This is no longer true even in the easiest cases of nonlinear optimization, as
can be seen on the following quadratically constrained problem in one
continuous variable:
\begin{displaymath}
  \begin{aligned}
    \max\quad & f(x) = x^4 \quad
    \st\quad & x^2 \leq 2.
  \end{aligned}
\end{displaymath}
Here the unique optimal solution is irrational ($x^*=\sqrt2$, with
$f(x^*)=4$), and so it does not have a finite binary encoding.  We ignore here
the possibilities of using a model of computation and complexity over the real
numbers, such as the celebrated Blum--Shub--Smale model
\cite{blum-shub-smale-1989}.  In the familiar Turing model of computation, we
need to resort to approximations.

In the example above it is clear that for every $\epsilon>0$, there exists a
rational~$x$ that is a \emph{feasible} solution for the problem and satisfies
$|x-x^*| < \epsilon$ (proximity to the optimal solution) or $|f(x) - f(x^*)| <
\epsilon$ (proximity to the optimal value).  However, in general we cannot
expect to find approximations by feasible solutions, as the following example
shows.
\begin{displaymath}
  \begin{aligned}
    \max\quad & f(x) = x \quad
    \st\quad & x^3 - 2x = 0.
  \end{aligned}
\end{displaymath}
(Again, the optimal solution is $x=\sqrt2$, but the closest rational feasible
solution is $x=0$.)  Thus, in the general situation, we will have to use the
following notion of approximation:

\begin{definition}
  An algorithm $\mathcal A$ is said to \emph{efficiently approximate} an
  optimization problem if, for every value of the input parameter
  $\epsilon>0$, it returns a rational vector~$\ve x$ (not necessarily
  feasible) with $\| \ve x-\ve x^*\| \leq \epsilon$, where $\ve x^*$ is an
  optimal solution, and the running time of~$\mathcal A$ is polynomial in the
  input encoding of the instance and in $\log 1/\epsilon$.
\end{definition}

\subsection{Approximation algorithms and schemes}

The polynomial dependence of the running time in $\log 1/\epsilon$, as defined
above, is a very strong requirement.  For many problems, efficient
approximation algorithms of this type do not exist, unless
$\mathrm{P}=\mathrm{NP}$.  The following, weaker notions of approximation are
useful; here it is common to ask for the approximations to be \emph{feasible
  solutions}, though.

\begin{definition}
\begin{enumerate}[\rm(a)]
\item An algorithm ${\cal  A}$
is an  \emph{${\epsilon}$-approximation algorithm}
for a maximization  problem with optimal cost $f_{\max}$,
if for each instance  of the problem of encoding length~$n$,
${\cal  A}$  runs in polynomial time in $n$ and returns
a feasible solution with cost   $f_{\rm {\cal  A}} $, such that
\begin{equation}\label{eq:epsilon-approx}
  f_{\rm {\cal  A}}    \geq (1-\epsilon) \cdot  f_{\max}.
\end{equation}
 \item A family of  algorithms ${\cal  A}_{\epsilon}$ is a \emph{polynomial time
     approximation scheme (PTAS)} if for every error parameter
   $\epsilon>0$, ${\cal  A}_{\epsilon}$ is an $\epsilon$-approximation algorithm
   \index{approximation!$\epsilon$-approximation}and its running time is
   polynomial in the size of the instance for every fixed
   $\epsilon$.
\item A family $\{\mathcal A_\epsilon\}_\epsilon$ of $\epsilon$-approximation
  algorithms is a \emph{fully polynomial time approximation scheme (FPTAS)} if
  the running time of~$\mathcal A_\epsilon$ is polynomial in the encoding size
  of the instance and $1/\epsilon$.
\end{enumerate}
\end{definition}

These notions of approximation are the usual ones in the domain of
combinatorial optimization.  It is clear that they are only useful when the
function~$f$ (or at least the maximal value $f_{\max}$) are non-negative.  For
polynomial or general nonlinear optimization problems, various authors
\cite{vavasis-1993,bellare-rogaway-1993,deklerk-laurent-parillo:ptas-polynomial-simplex}
have proposed to use a different notion of approximation, where we compare the
approximation error to the \emph{range} of the objective function on the
feasible region,
\begin{equation}
  \label{eq:deklerk-approx-1}
  \bigl| f_{\mathcal A} - f_{\max} \bigr| 
  \leq \epsilon \bigl| f_{\max} - f_{\min} \bigr|.
\end{equation}
(Here $f_{\min}$ denotes the minimal value of the function on the feasible region.)
It enables us to study objective functions that are
not restricted to be non-negative on the feasible region.  
In addition, this notion of approximation is invariant under shifting of the
objective function by a constant, and under exchanging minimization and
maximization.  On the other hand, it is not useful for optimization problems
that have an infinite range.
We remark that, when the objective function can take negative values on the
feasible region, \eqref{eq:deklerk-approx-1} is weaker
than~\eqref{eq:epsilon-approx}.  We will call approximation algorithms and
schemes with respect to this notion of approximation \emph{weak}.
This terminology, however, is not consistent in the literature; \cite{deKlerk2008}, for
instance, 
uses the notion \eqref{eq:deklerk-approx-1} without an additional attribute
and instead reserves the word \emph{weak} for approximation algorithms and schemes that give a guarantee on the absolute error:
\begin{equation}
  \label{eq:absolute-approx-1}
  \bigl| f_{\mathcal A} - f_{\max} \bigr| 
  \leq \epsilon. 
\end{equation}

\section{Incomputability}

Before we can even discuss the computational complexity of nonlinear
mixed-integer optimization, we need to be aware of fundamental incomputability
results that preclude the existence of \emph{any} algorithm to solve general
integer polynomial optimization problems.  

Hilbert's tenth problem asked for an algorithm to decide whether a given
multivariate polynomial $p(x_1,\dots,x_n)$ has an integer root, i.e., whether
the Diophantine equation
\begin{equation}
  p(x_1,\dots,x_n) = 0, \quad x_1,\dots,x_n\in\Z
\end{equation}
is solvable.  It was answered in the negative by Matiyasevich
\cite{matiyasevich-1970}, based on earlier work by Davis,
Putnam, and Robinson; see also \cite{matiyasevich-1993}.  A short
self-contained proof, using register machines, is presented in
\cite{jones-matiyasevich-1991}. 
\begin{theorem}
  \begin{enumerate}[\rm(i)]
  \item There does not exist an algorithm that, given polynomials~$p_1,\dots,p_m$, decides
    whether the system $p_i(x_1,\dots,x_n)=0$, $i=1,\dots,m$, has a solution
    over the integers.
  \item There does not exist an algorithm that, given a polynomial~$p$, decides
    whether $p(x_1,\dots,x_n)=0$ has a solution over the integers.
  \item There does not exist an algorithm that, given a polynomial~$p$, decides
    whether $p(x_1,\dots,x_n)=0$ has a solution over the non-negative integers
    $\Z_+ = \{0,1,2,\dots\}$.
  \item There does not exist an algorithm that, given a polynomial~$p$, decides
    whether $p(x_1,\dots,x_n)=0$ has a solution over the natural numbers $\N = \{1,2,\dots\}$.
  \end{enumerate}
\end{theorem}
These three variants of the statement are easily seen to be equivalent.  The
solvability of the system $p_i(x_1,\dots,x_n)=0$, $i=1,\dots,m$, is equivalent
to the solvability of $\sum_{i=1}^m p_i^2(x_1,\dots,x_n) = 0$.  Also, if
$(x_1,\dots,x_n)\in\Z^n$ is a solution of $p(x_1,\dots,x_n)=0$ over the
integers, then by splitting variables into their positive and negative parts,
$y_i=\max\{0,x_i\}$ and $z_i=\max\{0,-x_i\}$, clearly $(y_1, z_1; \dots;
y_n,z_n)$ is a non-negative integer solution of the polynomial equation
$q(y_1,z_1; \dots; y_n,z_n) := p(y_1-z_1,\dots,y_n-z_n) = 0$.  (A construction
with only one extra variable is also possible: Use the non-negative variables
$w = \max \{ |x_i| : x_i < 0 \}$ and $y_i := x_i + w$.)  
In the other
direction, using Lagrange's four-square theorem, any non-negative integer~$x$
can be represented as the sum $a^2+b^2+c^2+d^2$ with integers $a,b,c,d$.
Thus, if $(x_1,\dots,x_n)\in\Z_+^n$ is a solution over the non-negative
integers, then there exists a
solution~$(a_1,b_1,c_1,d_1;\dots;a_n,b_n,c_n,d_n)$ of the polynomial equation
$r(a_1,b_1,c_1,d_1;\dots;a_n,b_n,c_n,d_n) := p(a_1^2+b_1^2+c_1^2+d_1^2, \dots,
a_n^2+b_n^2+c_n^2+d_n^2)$.  The equivalence of the statement with non-negative
integers and the one with the natural numbers follows from a simple change of
variables. 

Sharper statements of the above incomputability result can be found in
\cite{jones-1982}.  All incomputability statements appeal to the classic
result by Turing~\cite{turing-1936} on the existence of recursively enumerable
(or listable) sets of natural numbers that are not recursive, such as the
halting problem of universal Turing machines.  
\begin{theorem}
  For the following \emph{universal pairs} $(\nu, \delta)$ 
  \begin{displaymath}
    (58,4),\dots, (38, 8), \dots, (21,96), \dots, (14, 2.0\times 10^5),\dots,(9, 1.638 \times 10^{45}),
  \end{displaymath}
  there exists a \emph{universal polynomial} $U(x; z,u,y; a_1,\dots,a_\nu)$ of
  degree $\delta$ in $4 + \nu$ variables, 
  i.e., for every recursively enumerable (listable) set $X$ there exist
  natural numbers $z$, $u$, $y$, such that 
  \begin{displaymath}
    x \in X \quad\iff\quad \exists a_1,\dots,a_\nu\in\N: U(x; z,u,y;
    a_1,\dots,a_\nu)=0. 
  \end{displaymath}
\end{theorem}%
Jones explicitly constructs these universal polynomials, using and extending
techniques by Matiyasevich.  Jones also constructs an explicit system of
quadratic equations in $4 + 58$ variables that is universal in the same sense.
The reduction of the degree, down to~$2$, works at the expense of introducing
additional variables; this technique goes back to Skolem~\cite{skolem-1938:diophant}.

In the following, we highlight some of the consequences of these
results.   Let $U$ be a universal polynomial corresponding to a
universal pair $(\nu, \delta)$, and let $X$ be a recursively enumerable set
that is not recursive, i.e., there does not exist any algorithm (Turing
machine) to decide whether a given $x$ is in~$X$.  
By the above theorem, there exist natural numbers $z$,
$u$, $y$ such that $x\in X$ holds if and only if the polynomial equation $U(x;
z,u,y; a_1,\dots,a_\nu)=0$ has a solution in natural numbers
$a_1,\dots,a_\nu$ (note that $x$ and $z,u,y$ are fixed parameters here).  This
implies: 
\begin{theorem}\label{thm:jones-consequences}\label{th:polyopt-incomputable}
  \begin{enumerate}[\rm (i)]
  \item Let $(\nu,\delta)$ be any of the universal pairs listed above.  Then
    there does not exist any algorithm that, given a polynomial~$p$ of degree
    at most~$\delta$ in $\nu$ variables, decides whether $p(x_1,\dots,x_n)=0$
    has a solution over the non-negative integers.
  \item In particular, there does not exist any algorithm that, given a
    polynomial~$p$ in at most~$9$ variables, decides whether
    $p(x_1,\dots,x_n)=0$ has a solution over the non-negative integers.
  \item There also does not
    exist any 
    algorithm that, given a polynomial~$p$ in at most~$36$ variables, decides
    whether $p(x_1,\dots,x_n)=0$ has a solution over the
    integers.
  \item There does not exist any algorithm that, given a
    polynomial~$p$ of degree at most~$4$, decides whether $p(x_1,\dots,x_n)=0$
    has a solution over the non-negative integers (or over the integers).
  \item There does not exist any algorithm that, given a system of quadratic
    equations in at most $58$~variables, decides whether it has a solution of
    the non-negative integers.
  \item There does not exist any algorithm that, given a system of quadratic
    equations in at most $232$~variables, decides whether it has a solution of
    the integers.
  \end{enumerate}
\end{theorem}
We remark that the bounds of $4\times 9=36$ and $4\times 58 = 232$ are most
probably not sharp; they are obtained by a straightforward application of the
reduction using Lagrange's theorem.\smallbreak

For integer polynomial optimization, this has the following fundamental
consequences.  First of all, 
Theorem~\ref{thm:jones-consequences} can be understood as a
statement on the feasibility problem of an integer polynomial optimization
problem.  Thus, the feasibility of an integer polynomial optimization problem
with a single polynomial constraint in 9 non-negative integer variables or 36
free integer variables is undecidable, etc.

If we wish to restrict our attention to \emph{feasible}
optimization problems, we can consider the problem of minimizing
$p^2(x_1,\dots,x_n)$ over the integers or non-negative integers and conclude
that unconstrained polynomial optimization in 9 non-negative integer or 36
free integer variables is undecidable.  We can also follow
Jeroslow~\cite{jeroslow-1973:quadratic-ip-uncomputable} and associate with an
arbitrary polynomial~$p$ in $n$ variables the optimization problem
\begin{displaymath}
  \begin{aligned}
    \min\quad& u \\
    \st\quad & (1-u)\cdot p(x_1,\dots,x_n) = 0,\\
    & u\in\Z_+,\quad \ve x\in\Z_+^n.
  \end{aligned}
\end{displaymath}
This optimization problem is always feasible and has the optimal solution
value~$0$ if and only if $p(x_1,\dots,x_n)=0$ is solvable, and $1$ otherwise.
Thus, optimizing \emph{linear forms} over one polynomial constraint in 10
non-negative integer variables is incomputable, and similar statements can be
derived from the other universal pairs above.
Jeroslow~\cite{jeroslow-1973:quadratic-ip-uncomputable} used the above program
and a degree reduction (by introducing additional variables) to prove the
following.
\begin{theorem}
  \label{th:quadopt-incomputable}
  The problem of minimizing a linear form over quadratic inequality
  constraints in integer variables is not computable; this still holds true
  for the subclass of problems that are feasible, and where the minimum value
  is either $0$ or $1$.
\end{theorem}

This statement can be strengthened by giving a bound on the number of integer
variables.


\section{Hardness and inapproximability}

All incomputability results, of course, no longer apply when finite bounds for
all variables are known; in this case, a trivial enumeration approach gives a
finite algorithm.  This is immediately the case when finite bounds for all
variables are given in the problem formulation, such as for 0-1 integer
problems. 

For other problem classes, even though finite bounds are not given, it is
possible to compute such bounds that either hold for all feasible solutions or
for an optimal solution (if it exists).  This is well-known for the case of
linear constraints, where the usual encoding length estimates of basic
solutions \cite{GroetschelLovaszSchrijver88} are available.  As we explain in
section~\ref{s:convex-min} below, such finite bounds can also be computed for
convex and quasi-convex integer optimization problems.

In other cases, algorithms to decide feasibility exist even though no finite
bounds for the variables are known.  An example is the case of single
Diophantine equations of degree 2, which are decidable using an algorithm by
Siegel \cite{siegel-1972:quadratic-forms}.  We discuss the complexity of this
case below.\smallbreak

Within any such computable subclass, we can ask the question of the
complexity. Below we discuss hardness results that come from the number
theoretic side of the problem~(section \ref{s:hardness-from-diophant}) and
those that come from the continuous optimization side
(section~\ref{s:hardness-from-nonlinear}. 

\subsection{Hardness results from quadratic Diophantine equations in fixed
  dimension}
\label{s:hardness-from-diophant}

The computational complexity of single quadratic Diophantine equations in
2~variables is already very interesting and rich in structure; we refer to to
the excellent paper by
Lagarias~\cite{lagarias-2006:cert-binary-quadratic-diophantine}.  Below we
discuss some of these aspects and their implications on optimization.

Testing primality of a number~$N$ is equivalent to deciding feasibility of the
equation
\begin{equation}\label{eq:primes}
  (x + 2)(y + 2) = N
\end{equation}
over the non-negative integers.  Recently, Agrawal, Kayal, and Saxena
\cite{agrawal-kayal-saxena:primes-p} showed that primality can be tested in
polynomial time.  However, the complexity status of \emph{finding} factors of
a composite number, i.e., finding a solution~$(x,y)$ of~\eqref{eq:primes}, is
still unclear.

The class also contains subclasses of NP-complete feasibility problems, 
such as the problem of deciding for given $\alpha,\beta,\gamma\in\N$ whether
there exist $x_1,x_2\in\Z_+$ with $\alpha x_1^2+\beta x_2 = \gamma$
\cite{manders-adleman:1978}.
On the other hand, the problem of
deciding for given $a,c\in\N$ whether there exist $x_1,x_2\in\Z$ with $ax_1x_2
+ x_2=c$, lies in $\mathrm{NP}\setminus \mathrm{coNP}$ unless
$\mathrm{NP}=\mathrm{coNP}$ \cite{adleman-manders:stoc-1977}.

The feasibility problem of the general class of quadratic Diophantine
equations in two (non-negative) variables was shown by
Lagarias~\cite{lagarias-2006:cert-binary-quadratic-diophantine} to be in NP.
This is not straightforward because minimal solutions can have an encoding
size that is exponential in the input size.  This can be seen in the case of
the so-called \emph{anti-Pellian equation} $x^2 - dy^2 = -1$.  Here Lagarias
\cite{lagarias-1980:complexity} proved that for all $d = 5^{2n+1}$, there
exists a solution, and the
solution with minimal binary encoding length has an encoding length of
$\Omega(5^n)$ (while the input is of encoding length $\Theta(n)$). 
(We remark that the special case of the anti-Pellian equation is in coNP, as
well.) \smallbreak

Related hardness results include the problem of quadratic congruences with a
bound, i.e., deciding for given $a,b,c\in\N$ whether there exists a positive
integer $x<c$ with $x^2\equiv a \pmod{b}$; this is the NP-complete problem AN1
in \cite{GareyJohnson79}.\smallbreak

From these results, we immediately get the following consequences on
optimization. 
\begin{theorem} \label{th:deg4-dim2-hard} 
  \begin{enumerate}[\rm(i)]
  \item The feasibility problem of quadratically constrained problems in $n=2$
    integer variables is NP-complete.
  \item The problems of computing a feasible (or optimal) solution of quadratically
    constrained problems in $n=2$ integer variables is not polynomial-time solvable
    (because the output may require exponential space).
  \item The feasibility problem of quadratically constrained problems in $n>2$
    integer variables is NP-hard (but it is unknown whether it is in NP).
  \item The problem of minimizing a degree-$4$ polynomial over the lattice
    points of a convex polygon (dimension $n=2$) is NP-hard.  
  \item The problem of finding the minimal value of a degree-$4$ polynomial
    over $\Z^2_+$ is NP-hard; writing down an optimal solution cannot be done
    in polynomial time.
  \end{enumerate}
\end{theorem}

However, the complexity of minimizing a quadratic form over the integer points in
polyhedra of fixed dimension is unclear, even in dimension~$n=2$.  
Consider the integer convex minimization problem
\begin{displaymath}
  \begin{aligned}
    \min\quad & \alpha x_1^2 + \beta x_2, \\
    \st\quad& x_1, x_2 \in\Z_+
  \end{aligned}
\end{displaymath}
for $\alpha,\beta\in\N$.  Here an optimal solution
can be obtained efficiently, as we explain in section~\ref{s:convex-min}; in fact, clearly
$x_1=x_2=0$ is the unique optimal solution.  On the other
hand, 
the  problem whether there exists a point $(x_1,x_2)$ of a prescribed
objective value~$\gamma = \alpha x_1^2 + 
\beta x_2$ is NP-complete (see above).  For indefinite quadratic forms, even
in dimension~$2$, nothing seems to be known.

In varying dimension, the convex quadratic maximization case, i.e., maximizing
positive definite quadratic forms is an NP-hard problem.  This is even true in
very restricted settings such as the problem to maximize $\sum_i {(\ve
  w_i^\top \ve x)^2}$ over $\ve x\in\{0,1\}^n$
\cite{onn:nonlinear-discrete-monograph}.  \smallbreak

\subsection{Inapproximability of nonlinear optimization in varying dimension}
\label{s:hardness-from-nonlinear}

Even in the pure continuous case, nonlinear optimization is known to be hard.
Bellare and Rogaway \cite{bellare-rogaway-1993,BellareRogaway1995} proved
the following inapproximability results 
using the theory of interactive proof systems.
\begin{theorem} Assume that $\mathrm P\neq\mathrm{NP}$.
  \begin{enumerate}[\rm (i)]
  \item 
    For any $\epsilon<\frac13$, there does not exist a polynomial-time
    weak $\epsilon$-approximation algorithm for the problem of (continuous)
    quadratic programming over polytopes.
  \item 
    There exists a constant~$\delta>0$ such that the problem of polynomial
    programming over polytopes does not have a polynomial-time weak
    $(1-n^{-\delta})$-approximation algorithm. 
  \end{enumerate}
\end{theorem}

Here the number $1-n^{-\delta}$ becomes arbitrarily close to~$0$ for
growing~$n$; note that a weak $0$-approximation algorithm is one that gives no guarantee
other than returning a feasible solution.

Inapproximability still holds for the special case of minimizing a quadratic
form over the cube $[-1,1]^n$ or over the standard simplex. 
In the case of the cube, inapproximability of the max-cut problem is used.
In the case of the standard simplex, it follows via the celebrated
Motzkin--Straus theorem \cite{motzkin-straus} from the inapproximability 
of the maximum stable set problem.  These are results by
H\aa{}stad~\cite{Hastad:inapprox97}; see also \cite{deKlerk2008}.




\section{Approximation schemes}

For important classes of optimization problems, while exact optimization is
hard, good approximations can still be obtained efficiently.


Many such examples are known in combinatorial settings.  As an example in
continuous optimization, we refer to the problem of maximizing homogeneous
polynomial functions of fixed degree over simplices.  Here de Klerk et
al.~\cite{deklerk-laurent-parillo:ptas-polynomial-simplex} proved a weak PTAS.
\smallbreak

Below we present a general result for mixed-integer polynomial optimization
over polytopes.

\subsection{Mixed-integer polynomial optimization in fixed dimension over
  linear constraints:  FPTAS and weak FPTAS} 
\label{s:fptas}

Here we consider the problem
\begin{equation}
\begin{aligned}
  \hbox{max/min}\quad & f(x_1,\dots,x_n)\\
  \hbox{subject to} \quad & A \mathbf x \leq \mathbf b\\
  & \mathbf x\in \R^{n_1}  \times \Z^{n_2},
\end{aligned} \label{eq:nonlinear-over-polyhedron}
\end{equation}
where $A$ is a rational matrix and $\mathbf b$ is a rational vector.  As we
pointed out above (Theorem \ref{th:deg4-dim2-hard}), optimizing degree-4
polynomials over problems with two integer variables ($n_1=0$, $n_2=2$) is
already a hard problem.  Thus, even when we fix the dimension, we cannot get a
polynomial-time algorithm for solving the optimization problem.  The best we
can hope for, even when the number of both the continuous and the integer
variables is fixed, is an approximation result.

We present here the FPTAS obtained by De Loera et
al. \cite{deloera-hemmecke-koeppe-weismantel:intpoly-fixeddim,DeloeraHemmeckeKoeppeWeismantel06,deloera-hemmecke-koeppe-weismantel:mixedintpoly-fixeddim-fullpaper},
which uses the ``summation method'' and the theory of \emph{short rational
  generating functions} pioneered by Barvinok \cite{Barvinok94,BarviPom}.  We
review the methods below; the FPTAS itself appears in Section \ref{s:fptas}.
An open question is briefly discussed at the end of this section.

\subsubsection{The summation method}

The summation method for optimization is the idea to use of elementary relation
\begin{equation}
  \max\{s_1, \dots, s_N \} = \lim_{k \rightarrow \infty} \sqrt[k]{s_1^k +
    \dots + s_N^k},
\end{equation}
which holds for any finite set $S=\{s_1,\dots,s_N\}$ of non-negative real
numbers.  This relation can be viewed as an approximation result for
$\ell_k$-norms.  Now if $P$ is a polytope and $f$ is an objective function
non-negative on $P\cap\Z^d$, let $\mathbf x^1,\dots,\mathbf x^{N}$ denote all the feasible
integer solutions in~$P\cap\Z^d$ and collect their objective function values
$s_i=f(\mathbf x^i)$
in a vector~$\mathbf s\in\Q^N$.  Then, comparing the unit balls of the $\ell_k$-norm
and the $\ell_\infty$-norm (Figure~\ref{fig:lp-norms}), we get the relation
\begin{displaymath}
  L_k :=
  {{N}}^{-1/k} \mathopen\| \mathbf s \mathclose\|_k \leq {\mathopen\| \mathbf s \mathclose\|_\infty}
  \leq \mathopen\| \mathbf s\mathclose\|_k
  =: U_k.
\end{displaymath}%
\begin{figure}[t]
  \centering
  $k=1$
  \ifpdf
    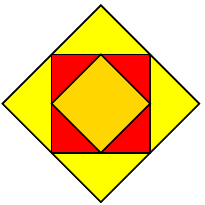
    \else
    \input{l1-norm.pstex_t}
    \fi
  \qquad
  $k=2$
  \ifpdf
    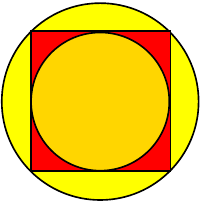
    \else
    \input{l2-norm.pstex_t}
    \fi
  \caption{Approximation properties of $\ell_k$-norms}
  \label{fig:lp-norms}
\end{figure}%
These estimates are independent of the function~$f$.  (Different estimates
that make use of the properties of~$f$,
and that are suitable also for the continuous case, can be obtained from the H\"older inequality;
see for instance \cite{baldoni-berline-deloera-koeppe-vergne:integration}.)

Thus, for obtaining a good approximation of the maximum, it suffices to solve
a summation problem of the polynomial function $h=f^k$ on $P\cap\Z^d$ for a
value of~$k$ that is large enough. Indeed, for $k= \bigl\lceil{(1+1/\epsilon)\log
  {{N}}}\bigr\rceil$, we obtain $U_k - L_k\leq \epsilon f(\mathbf x^{\max})$.  On the other
hand, this choice of~$k$ is polynomial in the input size (because $1/\epsilon$
is encoded in unary in the input, and $\log N$ is bounded by a polynomial in
the binary encoding size of the polytope~$P$).  Hence, when the dimension~$d$ is
fixed, we can expand the polynomial function~$f^k$ as a list of
monomials in polynomial time.

\subsubsection{Rational generating functions}

To solve the summation problem, one uses the technology of short rational
generating functions.  We explain the theory on a simple, one-dimensional example.  Let us consider the
set~$S$ of integers in the interval $P=[0,\dots,n]$. 
We associate with~$S$ the polynomial
\begin{math}
  g(P; z) = z^0 + z^1 + \dots + z^{n-1} + z^n;
\end{math}
i.e., every integer $\alpha\in S$ corresponds to a monomial~$z^\alpha$
with coefficient~$1$ in the polynomial~$g(P; z)$.
This polynomial is called the \emph{generating function} of~$S$ (or of~$P$).
From the viewpoint of computational complexity, this generating function is of
exponential size (in the encoding length of~$n$), just as an explicit list of
all the integers~$0$, $1$, \dots, $n-1$, $n$ would be.  However, we can observe that
$g(P; z)$ 
is a finite geometric series, so there exists
a simple summation formula that expresses
it 
in a much more compact way:
\begin{equation}
  \label{eq:1d-example-formula-summed}
  g(P; z) = z^0 + z^1 + \dots + z^{n-1} + z^n
  = \frac{1-z^{n+1}}{1-z}.
\end{equation}
The ``long'' polynomial has a ``short'' representation as a rational function.
The encoding length of this new formula is \emph{linear} in the encoding
length of~$n$.
On the basis of this idea, we can solve the summation problem.  Consider the generating
function of the interval~$P=[0,4]$,
\begin{displaymath}
  g(P;z) = z^0 + z^1 + z^2 + z^3 + z^4
  {= \frac1{1-z} - \frac{z^5}{1-z}}.
\end{displaymath}
We now apply the differential operator $ z \frac{\mathrm d}{\mathrm d
    z} $ and obtain
\begin{displaymath}
  \left(z \frac{\mathrm d}{\mathrm d z}\right) g(P;z)
  = {1} z^1 + {2} z^2 + {3} z^3 + {4} z^4
  {= \frac1{(1-z)^2} - \frac{-4z^5 + 5z^4}{(1-z)^2}}
\end{displaymath}
Applying the same differential operator again, we obtain
\begin{multline*}
  \left(z \frac{\mathrm d}{\mathrm d z}\right) \left( z
    \frac{\mathrm d}{\mathrm d z} \right) g(P;z)
  = {1} z^1 + {4} z^2 + {9}
  z^3 + {16} z^4\\
  {= \frac{z + z^2}{(1-z)^{3}}
    - \frac{25z^5 -39z^6+16 z^7}{(1-z)^{3}}}
\end{multline*}
We have thus evaluated the monomial function $h(\alpha)=\alpha^2$ for
$\alpha=0,\dots,4$; the results appear as the coefficients of the
respective monomials.  Substituting $z=1$ yields the desired sum
\begin{displaymath}
  \left.\left(z \frac{\mathrm d}{\mathrm d z}\right) \left( z
    \frac{\mathrm d}{\mathrm d z} \right) g(P;z) \right|_{z=1}
  = {1} + {4} + {9} + {16} = 30
\end{displaymath}
The idea now is to evaluate this sum instead by computing the limit of the
rational function for $z\to 1$,
\begin{displaymath}
  \sum_{\alpha=0}^4 \alpha^2 = \lim_{z\to1} \left[\frac{z + z^2}{(1-z)^{3}}
    - \frac{25z^5 -39z^6+16 z^7}{(1-z)^{3}}\right];
\end{displaymath}
this can be done using residue techniques.\smallbreak

We now present the general definitions and results.  Let $P\subseteq\R^d$ be a
rational polyhedron.  We first define its \emph{generating function} as the
\emph{formal Laurent series} $\tilde g(P; \mathbf z) = \sum_{{\bm{\alpha}}\in
  P\cap\Z^d} \mathbf z^{{\bm{\alpha}}} \in \Z[[z_1,\dots,z_d,
z_1^{-1},\dots,z_d^{-1}]]$, i.e., without any consideration of convergence
properties.  By convergence, one moves to a \emph{rational generating
  function} $g(P;\mathbf z)\in\Q(z_1,\dots,z_d)$.

The following breakthrough result was obtained by Barvinok in 1994.
\begin{theorem}[Barvinok \cite{Barvinok94}]
  Let $d$ be fixed.  There exists a polynomial-time algorithm for computing
  the generating function $g(P;\mathbf{z})$ of a polyhedron $P\subseteq\R^d$ given by rational
  inequalities in the form of a rational function
  \begin{equation}
    \label{eq:generating-function-0}
    g(P; \mathbf z) = \sum_{i\in I} \epsilon_i \frac{\mathbf
      z^{\mathbf a_i}}{\prod_{j=1}^d (1-\mathbf z^{\mathbf b_{ij}})}
    \quad\text{with $\epsilon_i\in\{\pm1\}$, $\mathbf a_i\in\Z^d$, and
  $\mathbf b_{ij}\in\Z^d$.}
  \end{equation}
\end{theorem}

\subsubsection{Efficient summation using rational generating functions}
\label{s:fptas}
Below we describe the theorems on the summation
method based on short rational generating functions, which appeared in \cite{deloera-hemmecke-koeppe-weismantel:intpoly-fixeddim,DeloeraHemmeckeKoeppeWeismantel06,deloera-hemmecke-koeppe-weismantel:mixedintpoly-fixeddim-fullpaper}.   Let $g(P;\mathbf z)$ be
the rational generating function of~$P\cap\Z^d$, computed using Barvinok's
algorithm.  By symbolically applying differential operators to~$g(P;\mathbf
z)$, we can compute a short rational function representation of the Laurent
polynomial
\begin{math}
  g(P,h; \mathbf z) = \sum_{{\bm{\alpha}}\in P \cap \Z^d} h({\bm{\alpha}}) \mathbf z^{{\bm{\alpha}}},
\end{math}
where each monomial $\mathbf z^{{\bm{\alpha}}}$ corresponding to an integer
point ${\bm{\alpha}}\in P\cap\Z^d$ has a coefficient that is the
value~$h({\bm{\alpha}})$.
As in the one-dimensional example above, we use the partial
differential operators $z_i \frac\partial{\partial z_i}$ for $i =1,\dots,d$ on
the short rational generating function.  
In fixed dimension, the size of the rational function expressions occuring in the symbolic
calculation can be bounded polynomially.  Thus one obtains the following result.

\begin{theorem}[\cite{deloera-hemmecke-koeppe-weismantel:intpoly-fixeddim}, Lemma 3.1]
  \label{operators}
  \begin{enumerate}[\rm(a)]
  \item Let $h(x_1,\dots,x_d) = \sum_{{\bm{\beta}}}c_{\bm{\beta}} \mathbf x^{{\bm{\beta}}}
    \in\Q[x_1,\dots,x_d]$ be a polynomial.  Define the differential operator
    $$D_h = h\left(z_1\frac{\partial}{\partial z_1},\dots, z_d\frac{\partial}{\partial
        z_d}\right) = \sum_{{\bm{\beta}}} c_{{\bm{\beta}}}
    \left(z_1\frac{\partial}{\partial z_1}\right)^{\beta_1}\dots
    \left(z_d\frac{\partial}{\partial z_d}\right)^{\beta_d}.$$
    Then $D_h$ maps the generating function $g(P;\mathbf z) = \sum_{{\bm{\alpha}}\in P\cap\Z^d}
    \mathbf z^{{\bm{\alpha}}}$ to the weighted generating function
    $(D_h g)(\mathbf z) = g(P, h; \mathbf z) = \sum_{{\bm{\alpha}}\in P\cap\Z^d} h({\bm{\alpha}}) \mathbf z^{{\bm{\alpha}}}$.
  \item
    Let the dimension $d$ be fixed.  Let $g(P;\mathbf z)$ be the Barvinok representation of the
    generating function $\sum_{{\bm{\alpha}}\in P
      \cap \Z^d}\mathbf z^{{\bm{\alpha}}}$ of $P\cap\Z^d$. Let $h\in\Q[x_1,\dots,x_d]$ be a
    polynomial, given as a list of monomials with rational coefficients~$c_{{\bm{\beta}}}$
    encoded in binary and exponents~${\bm{\beta}}$ encoded in unary. We can compute in
    polynomial time a Barvinok representation $g(P,h;\mathbf z)$ for the
    weighted generating function $\sum_{{\bm{\alpha}}\in P \cap \Z^d}
    h({\bm{\alpha}}) \mathbf z^{{\bm{\alpha}}}.$
  \end{enumerate}
\end{theorem}

Thus, we can implement the following algorithm in polynomial time (in fixed dimension).

\begin{algorithm}[Computation of bounds for the optimal value]~\smallskip
\label{Algorithm}

\noindent {\em Input:} A rational convex polytope $P \subset \R^d$;
a polynomial objective function $f \in \Q[x_1,\dots,x_d]$ 
that is non-negative over $P\cap\Z^d$,
given as a list of monomials with rational coefficients~$c_{{\bm{\beta}}}$
encoded in binary and exponents~${\bm{\beta}}$ encoded in unary;
an index~$k$, encoded in unary.\smallskip

\noindent {\em Output:} A lower bound~$L_k$ and an upper bound~$U_k$ for the maximal
function value $f^*$ of $f$ over $P\cap\Z^d$.
The bounds $L_k$ form a nondecreasing, the bounds $U_k$ a nonincreasing
sequence of bounds that both reach~$f^*$ in a finite number of steps.

\begin{enumerate}[\rm1.]
\item  Compute a short rational function expression for
  the generating function $g(P;\mathbf z)=\sum_{{\bm{\alpha}}\in P\cap\Z^d} \mathbf
  z^{{\bm{\alpha}}}$.  Using residue techniques, compute $|P \cap
  \Z^d|=g(P;\mathbf 1)$ from $g(P;\mathbf z)$.

\item Compute the polynomial~$f^k$ from~$f$.

\item From the rational function $ g(P;\mathbf z)$
  compute the rational function representation of $g(P,f^k;\mathbf z)$ of
  $\sum_{{\bm{\alpha}}\in P \cap \Z^d} f^k({\bm{\alpha}}) \mathbf z^{\bm{\alpha}}$ by
  Theorem \ref{operators}. Using residue techniques, compute
\[
L_k:=\left\lceil{\sqrt[k]{g(P,f^k;\mathbf 1)/g(P;\ve1)}}\,\right\rceil\;\;\;\text{and}\;\;\;
U_k:=\left\lfloor{\sqrt[k]{g(P,f^k;\mathbf 1)}}\right\rfloor.
\]
\end{enumerate}
\end{algorithm}
From the discussion of the convergence of the bounds, one then
obtains the following result.

\begin{theorem}[Fully polynomial-time approximation scheme]
  Let the dimension $d$ be fixed.  Let $P\subset\R^d$ be a rational convex polytope.
  Let $f$ be a polynomial with rational coefficients that is non-negative on
  $P\cap\Z^d$, given as a list of monomials with rational coefficients~$c_{{\bm{\beta}}}$
  encoded in binary and exponents~${\bm{\beta}}$ encoded in unary.
  \begin{enumerate}[\rm(i)]
\item Algorithm~\ref{Algorithm} computes the bounds $L_k$, $U_k$ in time polynomial in
  $k$, the input size of $P$ and $f$, and the total degree~$D$. The bounds
  satisfy 
$
U_k-L_k \leq f^* \cdot \left(\sqrt[k]{|P \cap \Z^d|}-1 \right).
$
\item For $k=(1+1/\epsilon)\log({|P \cap \Z^d|})$ (a number bounded by a
  polynomial in the input size),
  $L_k$ is a $(1-\epsilon)$-approximation to the optimal value $f^*$ and it
  can be computed in time polynomial in the input size, the total
  degree~$D$, and $1/\epsilon$. Similarly, $U_k$ gives a
  $(1+\epsilon)$-approximation to $f^*$.

\item With the same complexity,  by iterated bisection of $P$, we can also find
  a feasible solution $\mathbf x_\epsilon\in P\cap\Z^d$ with
  \begin{math}
    \bigl|f(\mathbf x_\epsilon) - f^*\bigr| \leq \epsilon f^*.
  \end{math}
\end{enumerate}
\end{theorem}

\subsubsection{Extension to the mixed-integer case by discretization}

The mixed-integer case can be handled by \emph{discretization} of the continuous
variables.  We illustrate on an example that one needs to be careful
to pick a sequence of discretizations that actually converges.
Consider the mixed-integer \emph{linear} optimization problem depicted in
Figure~\ref{fig:example-slice-not-fulldim},
whose feasible region consists of the point $(\frac12,1)$ and the segment
$\{\,(x,0): x\in[0,1]\,\}$.  The unique optimal solution
is $x=\frac12$, $z=1$.  Now consider
the sequence of grid approximations 
where $x\in \frac1m \Z_{\geq0}$.  For even~$m$, the unique optimal solution to the
grid approximation is $x=\frac12$, $z=1$.  However, for odd~$m$, the unique
optimal solution is $x=0$, $z=0$.  Thus the full sequence of the optimal
solutions to the grid approximations does not converge because it has two limit
points; see Figure~\ref{fig:example-slice-not-fulldim}.
\begin{figure}[t]
  \begin{minipage}[t]{.3\linewidth}
    \ifpdf
    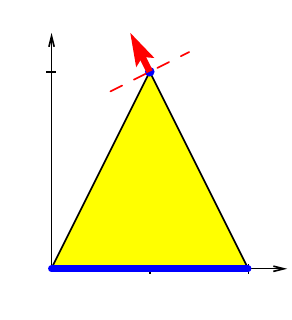
    \else
    \input{miptriangle.pstex_t}
    \fi
  \end{minipage}
  \quad
  \begin{minipage}[t]{.3\linewidth}
    \ifpdf
    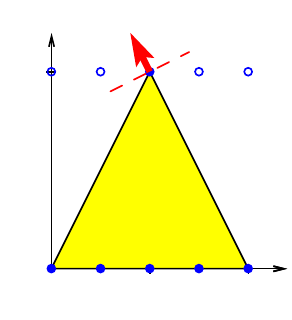
    \else
    \input{miptriangle-even.pstex_t}
    \fi
  \end{minipage}
  \quad
  \begin{minipage}[t]{.3\linewidth}
    \ifpdf
    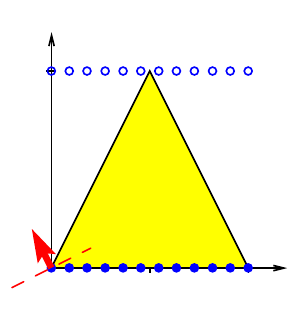
    \else
    \input{miptriangle-11.pstex_t}
    \fi
  \end{minipage}
  \caption{A mixed-integer linear optimization problem and a sequence of
    optimal solutions to grid problems with two limit points, for even $m$
    and for odd $m$}
  \label{fig:example-slice-not-fulldim}
\end{figure}

To handle polynomial objective functions that take arbitrary (positive and
negative) values, one can shift the objective function
by a large constant.  Then, to obtain a strong approximation
result, one iteratively reduces the constant by a factor.
Altogether we have the following result.

 \begin{theorem}[Fully polynomial-time approximation schemes]
   \label{th:mipo-fptas}
   Let the dimension $n=n_1+n_2$ be fixed.
   Let an optimization problem~\eqref{eq:nonlinear-over-polyhedron} of
   a polynomial function~$f$ over the mixed-integer points of a polytope~$P$
   and an error bound~$\epsilon$ be given, where
   \begin{enumerate}[\quad\rm({I}$_\bgroup 1\egroup$)]
   \item $f$ is given as a list of
     monomials with rational coefficients~$c_{{\bm{\beta}}}$ encoded in binary and
     exponents~${\bm{\beta}}$ encoded in unary,
   \item $P$ is given by rational inequalities in binary encoding,
   \item the rational number~$\frac1\epsilon$ is given in unary encoding.
   \end{enumerate}
   \begin{enumerate}[\rm(a)]
   \item There exists a fully polynomial time approximation scheme ({\small
       FPTAS}) for the \emph{maximization} problem for all
     polynomial functions $f(\mathbf x,\allowbreak \mathbf z)$
     that are \emph{non-negative} on the feasible region.
     That is, there exists a polynomial-time algorithm that,
     given the above data,
     computes a feasible solution $(\mathbf x_\epsilon, \mathbf
     z_\epsilon) \in P\cap\bigl(\R^{n_1}\times \Z^{n_2}\bigr)$ with
     \begin{displaymath}
       \bigl|f(\mathbf
       x_\epsilon, \mathbf z_\epsilon) - f(\mathbf x_{\max}, \mathbf z_{\max})\bigr|
       \leq \epsilon  f(\mathbf x_{\max},\mathbf z_{\max}).
     \end{displaymath}
   \item
     There exists a polynomial-time algorithm that,
     given the above data,
     computes a feasible solution $(\mathbf x_\epsilon, \mathbf
     z_\epsilon) \in P\cap\bigl(\R^{n_1}\times \Z^{n_2}\bigr)$ with
     \begin{displaymath}
       \bigl|f(\mathbf
       x_\epsilon, \mathbf z_\epsilon) - f(\mathbf x_{\max}, \mathbf z_{\max})\bigr|
       \leq \epsilon \bigl| f(\mathbf x_{\max},\mathbf z_{\max}) - f(\mathbf x_{\min}, \mathbf z_{\min}) \bigr|.
     \end{displaymath}
     In other words, this is a weak FPTAS.
   \end{enumerate}
 \end{theorem}

\subsubsection{Open question}

Consider the problem \eqref{eq:nonlinear-over-polyhedron} for a fixed number
$n_2$ of integer variables and a \emph{varying} number~$n_1$ of continuous
variables.  Of course, even with no integer variables present ($n_2=0$), this
is NP-hard and inapproximable.  On the other hand, if the objective
function~$f$ is \emph{linear}, the problem can be solved in polynomial time
using Lenstra's algorithm.  Thus it is interesting to consider  the problem
for an objective
function of restricted nonlinearity, such as 
\begin{displaymath}
  f(\ve x, \ve z) = g(\ve z) + \ve c^\top\ve x,
\end{displaymath}
with an arbitrary polynomial function~$g$ in the integer variables and a
linear form in the continuous variables.  The complexity (in particular the
existence of approximation algorithms) of this problem is an open question.

\section{Polynomial-time algorithms}

Here we study important special cases where polynomial-time algorithms can be
obtained.  We also include cases here where the algorithms efficiently
approximate the optimal solution to arbitrary precision, as discussed in 
section~\ref{s:encoding}.

\subsection{Fixed dimension: Continuous polynomial optimization}
\label{s:renegar}

Here we consider pure continuous polynomial optimization problems of the form
\begin{equation}
\begin{aligned}
  \hbox{min}\quad & f(x_1,\dots,x_n)\\
  \hbox{s.t.}\quad & g_1(x_1,\dots,x_n) \leq 0 \\
  & \quad\vdots \\
  & g_m(x_1,\dots,x_n) \leq 0 \\
  & \mathbf x\in\R^{n_1}.
\end{aligned} \label{eq:nonlinear-over-nonlinear-cont}
\end{equation}
When the dimension is fixed, this problem can be solved in polynomial time, in
the sense that there exists an algorithm that efficiently computes an
approximation to an optimal solution.  This follows from a much more general
theory on the computational complexity of approximating the solutions to
general algebraic and semialgebraic formulae over the real numbers by
Renegar~\cite{Renegar:1992:Approximating}, which we review in the following.
The bulk of this theory was developed in
\cite{Renegar:1992:CCGa,Renegar:1992:CCGb,Renegar:1992:CCGc}.  Similar results
appeared in \cite{heintz-roy-solerno-1990}; see also
\cite[Chapter 14]{basu-pollack-roy-2006:algo-real-algebraic-geom}).
One considers problems associated with logic
formulas of the form
\begin{equation}\label{s:quantified-formula}
  \mathrm{Q}_1\mathbf x^1\in\R^{n_1}\colon \ldots\ \mathrm{Q}_\omega\mathbf
  x^\omega\in\R^{n_\omega}\colon
  P(\mathbf y, \mathbf x^1,\dots,\mathbf x^\omega) 
\end{equation}
with quantifiers $\mathrm Q_i\in\{\exists,\forall\}$, where $P$ is a Boolean
combination of polynomial inequalities such as
\begin{displaymath}
  g_i(\mathbf y, \mathbf x^1,\dots,\mathbf x^\omega) \leq 0,\quad i=1,\dots,m,
\end{displaymath}
or using $\geq$, $<$, $>$, or $=$ as the relation.  Here $\ve y\in\R^{n_0}$ is a 
free (i.e., not quantified) variable.  Let $d\geq2$ be an upper
bound on the degrees of the polynomials $g_i$. A vector 
$\ve{\bar y}\in\R^{n_0}$ is called a \emph{solution} of this formula if the
formula~\eqref{s:quantified-formula} becomes a true logic sentence if we
set~$\ve y=\ve{\bar y}$.  Let $Y$ denote the set of all solutions. 
An \emph{$\epsilon$-approximate solution} is a vector~$\ve
y_\epsilon$ with $\|\ve{\bar y} - \ve y_\epsilon\| < \epsilon$ for some
solution~$\ve{\bar y}\in Y$.  

The following bound can be proved.  When the number~$\omega$ of ``blocks'' of
quantifiers (i.e., the number of alternations of the quantifiers $\exists$ and
$\forall$) is fixed, then the bound is singly exponential in the dimension.
\begin{theorem}\label{th:renegar-bound}
  If the formula~\eqref{s:quantified-formula} has only integer coefficients of
  binary encoding size at most~$\ell$, then every connected component of~$Y$
  intersects with the ball $\{ \| \ve y\| \leq r\}$, where 
  \begin{displaymath}
    \log r =
    \ell(md)^{2^{\Order(\omega)} n_0 n_1\cdots n_k}.
  \end{displaymath}
\end{theorem}

This bound is used in the following fundamental result, which gives a general
algorithm to compute $\epsilon$-approximate solutions to the
formula~\eqref{s:quantified-formula}. 
\begin{theorem}\label{th:renegar-approx}
  There exists an algorithm that, given numbers $0<\epsilon< r$ that are
  integral powers of~$2$ and a formula~\eqref{s:quantified-formula}, computes
  a set $\{\ve y_i\}_i$ of $(md)^{2^{\Order(\omega)} n_0 n_1\cdots n_k}$
  distinct
  $\epsilon$-approximate solutions of the formula with the property that for each connected
  components of $Y \cap \{ \| \ve y\| \leq r\}$ one of the $\ve y_i$ is within
  distance~$\epsilon$. 
  The algorithm runs in time 
  \begin{displaymath}
    (md)^{2^{\Order(\omega)} n_0n_1\dots n_k} \bigl(\ell
    +md+ \log \tfrac1\epsilon + \log r\bigr)^{\Order(1)}.
  \end{displaymath}
\end{theorem}

This can be applied to polynomial optimization problems as follows.  Consider
the formula
\begin{equation}
  \label{eq:optimum-formula}
  \begin{aligned}
    \forall \ve x\in\R^{n_1} : {}&g_1(\ve y) \leq 0 \ \wedge\ \cdots\ \wedge\  g_m(\ve
    y)\leq 0 \\
    &{}\wedge  \bigl[ g_1(\ve x) > 0 \ \vee\  \cdots \ \vee\  g_m(\ve x) > 0 \
    \vee\  f(\ve y)
      - f(\ve x) <0 \bigr],
  \end{aligned}
\end{equation}
this describes that~$\ve y$ is an optimal solution (all other solutions~$\ve
x$ are either infeasible or have a higher objective value).  Thus optimal
solutions can be efficiently approximated using the algorithm of
Theorem~\ref{th:renegar-approx}. 

\subsection{Fixed dimension: Convex and quasi-convex integer polynomial minimization}
\label{s:convex-min}

In this section we consider the case of the minimization of convex and
quasi-convex polynomials~$f$ over the mixed-integer points in convex regions
given by convex and quasi-convex polynomial functions~$g_1,\dots,g_m$:
\begin{equation}
\begin{aligned}
  \hbox{min}\quad & f(x_1,\dots,x_n)\\
  \hbox{s.t.}\quad & g_1(x_1,\dots,x_n) \leq 0 \\
  & \quad\vdots \\
  & g_m(x_1,\dots,x_n) \leq 0 \\
  & \mathbf x\in\R^{n_1} \times \Z^{n_2},
\end{aligned} \label{eq:nonlinear-over-nonlinear}
\end{equation}
Here a function $g\colon \R^n\to\R^1$ is called \emph{quasi-convex} if every
lower level set $L_\lambda = \{\, \ve x\in\R^n: g(x) \leq \lambda\,\}$ is a convex
set. 

The complexity in this setting is fundamentally
different from the general (non-convex) case.   One important aspect is 
that bounding results for the coordinates of optimal \emph{integer} solutions
exists, which are similar to the ones for continuous solutions in
Theorem~\ref{th:renegar-bound} above. 
For the case of
convex functions, these bounding results were obtained by
\cite{khachiyan:1983:polynomial-programming, tarasov-khachiyan-1980}.  An
improved bound was obtained by
\cite{bank-krick-mandel-dolerno-1991,bank-heintz-krick-mandel-solerno-1993},
which also handles the more general case of quasi-convex polynomials.  
This bound follows from the efficient theory of quantifier elimination over
the real numbers that we referred to in section~\ref{s:renegar}.
\begin{theorem}\label{th:encoding-length-convex-optimum}
  Let $f, g_1,\dots,g_m \in\Z[x_1,\dots,x_n]$ be quasi-convex polynomials
  of degree at most~$d\geq2$, whose coefficients have a binary encoding length
  of at most~$\ell$.  Let
  \begin{displaymath}
    F = \bigl\{\, \mathbf x\in\R^n : g_i(\mathbf x) \leq  0\quad
    \text{for $i=1,\dots,m$} \,\bigr\}
  \end{displaymath}
  be the (continuous) feasible region.  If the integer minimization
  problem
  \begin{math}
    \min\{\, f(\mathbf x): \mathbf x\in F\cap\Z^n \,\}
  \end{math}
  is bounded, there
  exists a radius~$R\in\Z_+$ of binary encoding length at most $(md)^{\mathrm
    O(n)} \ell$ such that
  \begin{displaymath}
    \min\bigl\{\, f(\mathbf x): \mathbf x\in F\cap\Z^n \,\bigr\}
    = \min\bigl\{\, f(\mathbf x): \mathbf x\in F\cap\Z^n, \ \mathopen\| \mathbf x
    \mathclose\| \leq R \,\bigr\}.%
  \end{displaymath}%
\end{theorem}

Using this finite bound, a trivial enumeration algorithm can find an optimal
solution (but not in polynomial time, not even in fixed dimension).  
Thus the incomputability result for
integer polynomial optimization (Theorem \ref{th:polyopt-incomputable}) does
not apply to this case.

The unbounded case can be efficiently detected in the case of quasi-convex
polynomials; see \cite{bank-krick-mandel-dolerno-1991} and
\cite{obuchowska-2008:boundedness-convex-integer}, the latter of which also
handles the case of ``faithfully convex'' functions that are not polynomials.
\smallbreak

\label{s:convex-min-fixed-dim}

In fixed dimension, the problem of convex integer minimization can be solved
efficiently
using variants of Lenstra's algorithm \cite{Lenstra83} for
integer programming.  
  \begin{figure}[t]
    \centering
    \ifpdf
    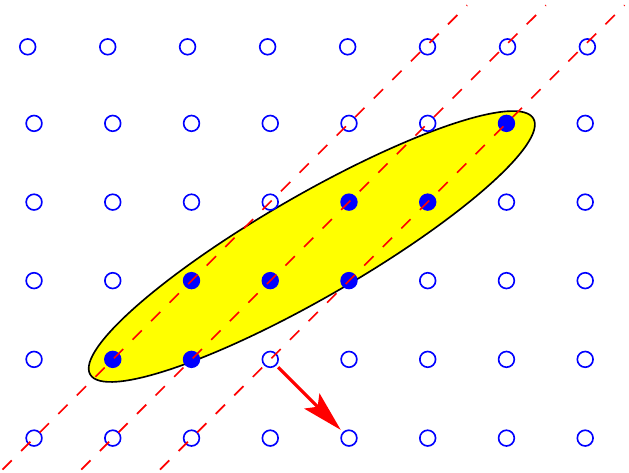
    \else
    \input{branching.pstex_t}
    \fi
    \caption{Branching on hyperplanes corresponding to approximate lattice
      width directions of the feasible region in a Lenstra-type algorithm}
    \label{fig:branching}
  \end{figure}%
Lenstra-type algorithms are algorithms for solving feasibility problems. 
We consider a family of feasibility problems associated
with the optimization problem, 
\begin{equation}
  \exists \ve x\in F_\alpha\cap\Z^n
  \quad\text{where}\quad F_\alpha = \{\, \ve x\in F: f(\ve x) \leq \alpha \,\}
  \quad\text{for $\alpha\in\Z$}.
\end{equation}
If bounds for $f(\ve x)$ on the feasible regions of polynomial binary encoding size
are known, a polynomial time algorithm for this feasibility problem can be
used in a binary search to solve the optimization problem in polynomial time.
Indeed, when the dimension~$n$ is fixed, the bound~$R$
given by Theorem \ref{th:encoding-length-convex-optimum} has a binary encoding
size that is bounded polynomially by the input data.  

A Lenstra-type algorithm uses branching on hyperplanes
(Figure~\ref{fig:branching}) to obtain polynomial time complexity in fixed
dimension.  Note that only the binary encoding size of the bound $R$, but not
$R$ itself, is bounded polynomially.  Thus, multiway branching on the values
of a single variable~$x_i$ will create an exponentially large number of
subproblems.  Instead, a Lenstra-type algorithm computes a primitive lattice
vector $\ve w\in\Z^n$ such that there are only few lattice hyperplanes $\ve
w^\top\ve x = \gamma$ (with $\gamma\in\Z$) that can have a nonempty intersection
with~$F_\alpha$.  The \emph{width} of $F_\alpha$ in the direction~$\ve w$, defined as
\begin{equation}
  \max\{\, \ve w^\top\ve x: x\in F_\alpha \,\} 
  - \min\{\, \ve w^\top\ve x: x\in F_\alpha \,\}
\end{equation}
essentially gives the number of these lattice points.  A \emph{lattice width
  direction} is a minimizer of the width among all directions $\ve w\in\Z^n$,
the \emph{lattice width} the corresponding width.  Any polynomial bound on the
width will yield a polynomial-time algorithm in fixed dimension.

Exact and approximate lattice width directions~$\ve w$ can be constructed
using geometry of numbers 
techniques.  We refer to the excellent tutorial
\cite{eisenbrand:50-years-ip-geom-num-chapter} and the classic references
cited therein.  The key to dealing with the feasible region~$F_\alpha$ is to
apply ellipsoidal rounding.  By applying the shallow-cut ellipsoid method
(which we describe in more detail below), one finds concentric proportional inscribed and
circumscribed ellipsoids that differ by some factor $\beta$ that only depends
on the dimension~$n$.  Then any $\eta$-approximate lattice width direction for
the ellipsoids gives a $\beta\eta$-approximate lattice width direction
for~$F_\alpha$.  Lenstra's original algorithm now uses an LLL-reduced basis of
the lattice~$\Z^n$ with respect to a norm associated with the ellipsoid; the last basis
vector then serves as an approximate lattice width direction.

The first algorithm of this kind for convex integer minimization
was announced by Khachiyan \cite{khachiyan:1983:polynomial-programming}.
In the following we present the variant of Lenstra's algorithm due to Heinz
\cite{heinz-2005:integer-quasiconvex}, which seems to yield the best 
complexity bound for the problem published so far.  The complexity result is the following.
\begin{theorem}\label{th:heinz-complexity}
  Let $f, g_1,\dots,g_m \in\Z[x_1,\dots,x_n]$ be quasi-convex polynomials of
  degree at most~$d\geq2$, whose coefficients have a binary encoding length of
  at most~$\ell$.  There exists an algorithm running in time $m
  \ell^{\mathrm{O}(1)} d^{\mathrm{O}(n)} 2^{\mathrm{O}(n^3)}$ that computes a
  minimizer~$\mathbf x^*\in\Z^n$ of the problem~\eqref{eq:nonlinear-over-nonlinear}
  or reports that no minimizer exists.  If the algorithm outputs a
  minimizer~$\mathbf x^*$, its binary encoding size is $\ell
  d^{\mathrm{O}(n)}$.
\end{theorem}
\par
We remark that the complexity guarantees can be improved dramatically by
combining Heinz' technique with more recent variants of Lenstra's algorithm
that rely on the fast computation of shortest vectors
\cite{hildebrand:faster-heinz}.

A complexity result of greater generality was presented by
Khachiyan and Porkolab~\cite{khachiyan-porkolab:00}.  It covers the case of
minimization of convex polynomials over the integer points in convex
semialgebraic sets given by \emph{arbitrary} (not necessarily quasi-convex)
polynomials.
\begin{theorem}\label{th:khachiyan-porkolab-complexity}
  Let $Y\subseteq\R^{n_0}$ be a convex set given by
  \begin{displaymath}
    Y = \bigl\{\, \mathbf y\in\R^{n_0} :
    \mathrm{Q}_1\mathbf x^1\in\R^{n_1}\colon \cdots\ \mathrm{Q}_\omega\mathbf
    x^\omega\in\R^{n_\omega}\colon
    P(\mathbf y, \mathbf x^1,\dots,\mathbf x^\omega) \,\bigr\}
  \end{displaymath}
  with quantifiers $\mathrm Q_i\in\{\exists,\forall\}$, where $P$ is a Boolean
  combination of polynomial inequalities
  \begin{displaymath}
    g_i(\mathbf y, \mathbf x^1,\dots,\mathbf x^\omega) \leq 0,\quad i=1,\dots,m
  \end{displaymath}
  with degrees at most~$d\geq2$ and coefficients of binary encoding size at
  most~$\ell$.  There exists an algorithm for solving the problem
  \begin{math}
    \min \{\, y_{n_0} : \mathbf y \in Y\cap\Z^{n_0} \,\}
  \end{math}
  in time $\ell^{\mathrm{O}(1)} (md)^{\mathrm{O}(n_0^4) \prod_{i=1}^\omega \mathrm{O}(n_i)}$.
\end{theorem}

When the dimension~$n_0+n_1+\dots+n_\omega$ is fixed, the algorithm runs in
polynomial time.  For the case of convex minimization where the feasible region is described by
convex polynomials, the complexity bound of
Theorem \ref{th:khachiyan-porkolab-complexity}, however, translates to $\ell^{\mathrm{O}(1)}
m^{\mathrm{O}(n^2)} d^{\mathrm{O}(n^4)}$, which is worse than the bound of
Theorem \ref{th:heinz-complexity}
\cite{heinz-2005:integer-quasiconvex}.\medbreak

In the remainder of this subsection, we describe the ingredients of the
variant of Lenstra's algorithm due to Heinz.  The algorithm starts out by
``rounding'' the feasible region, by applying the shallow-cut ellipsoid method
to find proportional inscribed and circumscribed ellipsoids.  It is well-known
\cite{GroetschelLovaszSchrijver88} that the shallow-cut ellipsoid method only
needs an initial circumscribed ellipsoid that is ``small enough'' (of
polynomial binary encoding size -- this follows from
Theorem \ref{th:encoding-length-convex-optimum}) and an implementation of a
\emph{shallow separation oracle}, which we describe below.


For a positive-definite matrix~$A$ we denote by $\mathcal{E}(A,\mathbf{\hat x})$ the ellipsoid $\{\, \mathbf
x\in\R^n : {(\mathbf x - \mathbf{\hat x})}^\top A (\mathbf x - \mathbf{\hat x}) \leq 1 \,\}$.

\begin{lemma}[Shallow separation oracle]
  Let $g_0,\dots,g_{m+1}\in\Z[\mathbf x]$ be quasi-convex polynomials of degree at
  most~$d$, the binary encoding sizes of whose coefficients are at most~$r$.
  Let the (continuous) feasible region~$F = \{\,\mathbf x\in\R^n :
  g_i(\mathbf x) < 0\,\}$ be contained in the ellipsoid~$\mathcal{E}(A,\mathbf{\hat x})$,
  where $A$ and $\mathbf{\hat x}$ have binary encoding size at most~$\ell$.
  There exists an algorithm with running time $m (lnr)^{\mathrm{O}(1)}
  d^{\mathrm{O}(n)}$ that outputs
  \begin{enumerate}[\rm(a)]
  \item ``true'' if
    \begin{equation}\label{eq:tough-ellipsoid}
      \mathcal{E}((n+1)^{-3} A, \mathbf{\hat x}) \subseteq F \subseteq \mathcal{E}(A, \mathbf{\hat
        x});
    \end{equation}
  \item otherwise, a vector $\mathbf c\in\Q^n\setminus\{\mathbf0\}$ of binary encoding
    length $(l+r) (dn)^{\mathrm{O}(1)}$ with
    \begin{equation}\label{eq:shallow-cut}
      F \subseteq \mathcal{E}(A, \mathbf{\hat x}) \cap
      \bigl\{\, \mathbf x \in\R^n : \mathbf c^\top (\mathbf x - \mathbf{\hat x}) \leq \tfrac1{n+1}
      (\mathbf c^\top A\mathbf c)^{1/2} \,\bigr\}.
    \end{equation}
  \end{enumerate}
\end{lemma}
\begin{proof}
  We give a simplified sketch of the proof, without hard complexity estimates.
  By applying an affine transformation to $F\subseteq\mathcal{E}(A,\mathbf{\hat x})$, we
  can assume that $F$ is contained in the unit ball~$\mathcal{E}(I,\mathbf0)$.  Let us
  denote as usual
  by $\mathbf e_1,\dots,\mathbf e_n$ the unit vectors and by $\mathbf e_{n+1},\dots,\mathbf
  e_{2n}$ their negatives.  The algorithm first constructs numbers
  $\lambda_{i1}, \dots, \lambda_{id} > 0$ with
  \begin{equation}\label{eq:lambda-bounds}
    \frac{1}{n+\frac32} < \lambda_{i1} < \dots < \lambda_{id} < \frac1{n+1}
  \end{equation}
  and the corresponding point sets
  \begin{math}
    B_i = \{\, \mathbf x_{ij} := \lambda_{ij} \mathbf
    e_i : j=1,\dots,d\,\};
  \end{math}
  see Figure \ref{fig:shallowcut-0}\,(a).  The choice of the bounds
  \eqref{eq:lambda-bounds} for~$\lambda_{ij}$ will ensure that we either find
  a large enough inscribed ball for~(a) or a deep enough cut for~(b).
  \begin{figure}[t]
    \centering
    (a)\hspace{-.5cm}
    \ifpdf
    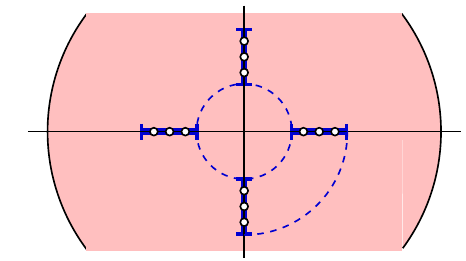
    \else
    \input{shallowcut-0.pstex_t}
    \fi
    \quad(b)\hspace{-.5cm}
    \ifpdf
    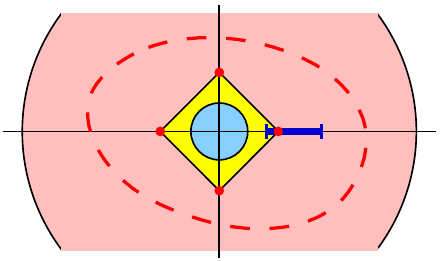
    \else
    \input{shallowcut.pstex_t}
    \fi
    \caption{The implementation of the shallow separation oracle.
      \textbf{(a)}~Test points $\mathbf x_{ij}$ in the circumscribed ball~$\mathcal{E}(1,\mathbf0)$.
      \textbf{(b)}~Case~I: All
      test points $\mathbf x_{i1}$ are (continuously) feasible; so their convex hull
      (a~cross-polytope) and its inscribed ball $\mathcal{E}((n+1)^{-3},\mathbf0)$ are
      contained in the (continuous) feasible region~$F$.
    }
    \label{fig:shallowcut-0}
    \label{fig:shallowcut}
  \end{figure}%
  Then the algorithm determines the (continuous) feasibility
  of the center~$\mathbf0$ and the $2n$ innermost points $\mathbf x_{i,1}$.\smallbreak

  \emph{Case~I.}
  If $\mathbf x_{i,1}\in F$ for $i=1,\dots, 2n$, then the cross-polytope
  $\mathop{\mathrm{conv}}\{\,\mathbf x_{i,1} : i = 1,\dots,2n\,\}$ is contained in~$F$; see
  Figure \ref{fig:shallowcut}\,(b).
  An easy
  calculation shows that the ball $\mathcal{E}((n+1)^{-3},\mathbf0)$ is contained in the
  cross-polytope and thus in~$F$; see Figure \ref{fig:shallowcut}.  Hence the
  condition in~(a) is satisfied and the algorithm outputs ``true''.\smallbreak

  \emph{Case~II.}  We now discuss the case when the center~$\mathbf 0$ violates a
  polynomial inequality $g_0(\mathbf x)<0$ (say).  Let $F_0 = \{\, \mathbf
  x\in\R^n: g_0(\mathbf x)<0\,\}\supseteq F$.  Due to convexity of~$F_0$, for all
  $i=1,\dots,n$, one set of each pair $B_i\cap F_0$ and $B_{n+i}\cap F_0$ must be
  empty; see~Figure~\ref{fig:shallowcut-2b}\,(a).  Without loss of generality, let us
  assume $B_{n+i}\cap F_0=\emptyset$ for all~$i$.
  \begin{figure}[t]
    \centering
    (a)\hspace{-.5cm}
    \ifpdf
    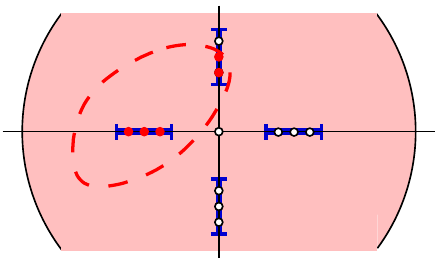
    \else
    \input{shallowcut-2b.pstex_t}
    \fi
    \quad(b)\hspace{-.5cm}
    \ifpdf
    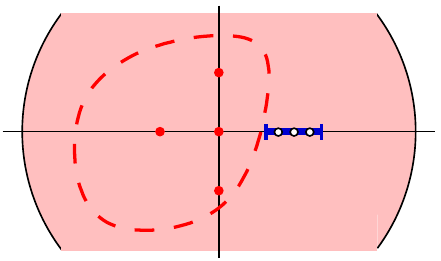
    \else
    \input{shallowcut-2a.pstex_t}
    \fi
    \caption{The implementation of the shallow separation oracle.  \textbf{(a)}~Case~II:
      The center~$\mathbf 0$ violates a polynomial inequality $g_0(\mathbf x)<0$
      (say).  Due to convexity, for all $i=1,\dots,n$, one set of each pair
      $B_i\cap F$ and $B_{n+i}\cap F$ must be empty.
      \textbf{(b)}~Case~III: A
      test point $\mathbf x_{k1}$ is infeasible, as it violates an
      inequality $g_0(\mathbf x)<0$ (say).  However, the center~$\mathbf 0$ is
      feasible at least for this inequality.}
    \label{fig:shallowcut-2a}
    \label{fig:shallowcut-2b}
  \end{figure}%
  We can determine whether a $n$-variate polynomial function of known maximum
  degree~$d$ is constant by evaluating it on $(d+1)^n$ suitable points (this
  is a consequence of the Fundamental Theorem of Algebra).  For our case of
  quasi-convex polynomials, this can be improved; indeed, it suffices to test
  whether the gradient $\nabla g_0$ vanishes on the $nd$ points in the set
  $B_1\cup\dots\cup B_{n}$.  If it does, we know that~$g_0$ is constant, thus
  $F=\emptyset$, and so can we return an arbitrary vector~$\mathbf c$.  Otherwise,
  there is a point $\mathbf x_{ij}\in B_i$ with $\mathbf c:= \nabla g_0(\mathbf x_{ij})\neq\mathbf0$;
  we return this vector as the desired normal vector of a shallow cut.
  Due to the choice of $\lambda_{ij}$ as a number smaller than~$\frac1{n+1}$,
  the cut is deep enough into the ellipsoid~$\mathcal{E}(A,\mathbf{\hat x})$,
  so that~\eqref{eq:shallow-cut} holds.
  \smallbreak

  \emph{Case~III.}  The remaining case to discuss is when $\mathbf0 \in F$ but
  there exists a $k \in\{1,\dots,2n\}$ with $\mathbf x_{k,1}\notin F$.  Without
  loss of generality, let $k=1$, and let $\mathbf x_{1,1}$ violate the polynomial
  inequality $g_0(\mathbf x)<0$, i.e., $g_0(\mathbf x_{1,1})\geq0$; see
  Figure \ref{fig:shallowcut-2a}\,(b).
  We consider the univariate polynomial
  $\phi(\lambda) = g_0(\lambda \mathbf e_i)$.  We have $\phi(0) = g_0(\mathbf0)<0$
  and $\phi(\lambda_{1,1}) \geq 0$, so $\phi$ is not constant.  Because $\phi$ has
  degree at most $d$, its derivative~$\phi'$ has degree at most~$d-1$, so
  $\phi'$ has at most $d-1$ roots.  Thus, for at least one of the $d$ different values
  $\lambda_{1,1}, \dots, \lambda_{1,d}$, say $\lambda_{1,j}$, we must have
  $\phi'(\lambda_{1,j})\neq0$.  This implies that $\mathbf c := \nabla g_0(\mathbf
  x_{1,j}) \neq\mathbf0$.  By convexity, we have $\mathbf x_{1,j}\notin F$, so we can
  use $\mathbf c$ as the normal vector of a shallow cut.
\end{proof}

By using this oracle in the shallow-cut ellipsoid method, one obtains
the following result.
\begin{corollary}\label{th:shallow-cut-ellipsoid-method}
  Let $g_0,\dots,g_{m}\in\Z[\mathbf x]$ be quasi-convex polynomials of degree at
  most~$d\geq2$.  Let the (continuous) feasible region~$F = \{\,\mathbf x\in\R^n :
  g_i(\mathbf x) \leq 0\,\}$ be contained in the ellipsoid~$\mathcal{E}(A_0,\mathbf0)$, given
  by the positive-definite matrix~$A_0\in\Q^{n\times n}$.  Let
  $\epsilon\in\Q_{>0}$ be given.  Let the entries of $A_0$ and the
  coefficients of all monomials of $g_0,\dots,g_{m}$ have binary encoding size
  at most~$\ell$.

  There exists an algorithm with running time $m (\ell
  n)^{\mathrm{O}(1)} d^{\mathrm{O}(n)}$ that computes a positive-definite
  matrix~$A\in\Q^{n\times n}$ and a point~$\mathbf{\hat x}\in\Q^n$ with
  \begin{enumerate}[\rm(a)]
  \item either $\mathcal{E}((n+1)^{-3} A, \mathbf{\hat x}) \subseteq F \subseteq \mathcal{E}(A, \mathbf{\hat x})$
  \item or $F \subseteq \mathcal{E}(A, \mathbf{\hat x})$ and $\mathop{\mathrm{vol}} \mathcal{E}(A, \mathbf{\hat x})
    <\epsilon$.
  \end{enumerate}
\end{corollary}

Finally, there is a lower bound for the volume of a continuous feasible
region~$F$ that can contain an integer point.

\begin{lemma}
  Under the assumptions of Corollary~\ref{th:shallow-cut-ellipsoid-method},
  if $F\cap\Z^n\neq \emptyset$, 
  there exists an $\epsilon\in\Q_{>0}$ of binary
  encoding size $\ell (dn)^{\mathrm{O}(1)}$ with $\mathop{\mathrm{vol}} F>\epsilon$.
\end{lemma}

On the basis of these results, one obtains a Lenstra-type algorithm for the
decision version of the convex integer minimization problem with the desired
complexity.  By applying binary search, the optimization problem can be
solved, which provides a proof of Theorem \ref{th:heinz-complexity}.

\subsection{Fixed dimension: Convex integer maximization}
\label{s:convex-max}
\label{s:convex-max-fixed-dim}

Maximizing a convex function over the integer points in a polytope in fixed
dimension can be done in polynomial time.  To see this, note that the optimal
value is taken on at a vertex of the convex hull of all feasible integer
points. But when the dimension is fixed, there is only a polynomial number of
vertices, as Cook et al.~\cite{cook-hartmann-kannan-mcdiarmid-1992} showed.
\begin{theorem}
  Let $P = \{\, \mathbf x\in\R^n : A\mathbf x\leq\mathbf b\,\}$ be a rational polyhedron
  with $A\in\Q^{m\times n}$ and let $\phi$ be the largest binary encoding size
  of any of the rows of the system~$A\mathbf x\leq\mathbf b$.  Let $P^{\mathrm{I}} =
  \mathop{\mathrm{conv}}(P\cap\Z^n)$ be the integer hull of~$P$.  Then the number of vertices
  of~$P^{\mathrm{I}}$ is at most $2 m^n{(6n^2\phi)}^{n-1}$.
\end{theorem}
Moreover, Hartmann \cite{hartmann-1989-thesis} gave an algorithm for
enumerating all the vertices, which runs in polynomial time in fixed
dimension.

By using Hartmann's algorithm, we can therefore compute all the vertices of
the integer hull~$P^{\mathrm{I}}$, evaluate the convex objective function on
each of them and pick the best.  This simple method already provides a
polynomial-time algorithm.


\section{Strongly polynomial-time algorithms: Submodular function minimization}

In important specially structured cases, even strongly polynomial-time
algorithms are available.  The probably most well-known case is that of
submodular function minimization.  We briefly present the most recent
developments below.


\newcommand{\TimeEval}{T_{\text{eval}}}

Here we consider the important problem of \emph{submodular function
  minimization}.  This class of problems consists of unconstrained 0/1
programming problems
\begin{displaymath}
  \min f(\ve x) : \ve x\in\{0,1\}^n,
\end{displaymath}
where the function~$f$ is submodular, i.e., $$f(\ve x) + f(\ve y) \geq
f(\max\{\ve x,\ve y\}) + f(\min\{\ve x,\ve y\}).$$
Here $\max$ and $\min$ denote the componentwise maximum and minimum of the
vectors, respectively.

The fastest algorithm known for submodular function minimization seems to be
by Orlin \cite{orlin-2009:faster-submodular}, who gave a strongly
polynomial-time algorithm of running time $\Order(n^5 \TimeEval + n^6)$, where
$\TimeEval$ denotes the running time of the evaluation oracle.  The algorithm
is ``combinatorial'', i.e., it does not use the ellipsoid method.  This
complexity bound simultaneously improved that of the fastest strongly
polynomial-time algorithm using the ellipsoid method, of running time
$\tilde\Order(n^5 \TimeEval + n^7)$ (see
\cite{mccormick-2005:submodular-in-handbook}) and the fastest
``combinatorial'' strongly polynomial-time algorithm by Iwata
\cite{iwata-2003:faster-scaling-submodular}, of running time $\Order((n^6 \TimeEval
+ n^7)\log n)$.  We remark that the fastest polynomial-time algorithm, by
Iwata \cite{iwata-2003:faster-scaling-submodular}, runs in
$\Order((n^4\TimeEval + n^5)\log M)$, where $M$ is the largest function value.
We refer to the recent survey by Iwata \cite{iwata-2008:submodular-survey},
who reports on the developments that preceded Orlin's algorithm
\cite{orlin-2009:faster-submodular}.

For the special case of \emph{symmetric} submodular function minimization,
i.e., $f(\ve x) = f(\ve 1-\ve x)$, Queyranne
\cite{queyranne-1998:symmetric-submodular} presented an algorithm of running
time $\Order(n^3 \TimeEval)$.

\section*{Acknowledgments.}
The author wishes to thank the referees, in particular for their comments on the
presentation of the Lenstra-type algorithm, and his student Robert Hildebrand
for a subsequent discussion about this topic.

\clearpage

\bibliography{iba-bib,nl-and-eng,barvinok,weismantel,hemmecke,../50-years-chapter/jon,./biblio,../rat/pisa-papers/biblio}
\bibliographystyle{siam}

\clearpage

\end{document}